\documentclass[12pt]{amsart}
\usepackage[english]{babel}
\usepackage{graphicx}
\usepackage{framed}
\usepackage[normalem]{ulem}
\usepackage{amsmath}
\usepackage{amsthm}
\usepackage{amssymb}
\usepackage{amsfonts}
\usepackage{enumerate}
\usepackage{comment}
\usepackage[utf8]{inputenc}
\usepackage[top=1 in,bottom=1 in, left=1 in, right=1 in]{geometry}
\usepackage{amsmath}
\usepackage{relsize}

\newcommand{\val}{\text{val}}
\newcommand{\rk}{\text{rk}}

\newcommand{\Z}{\mathbb{Z}}
\newcommand{\GL}{\operatorname{GL}}
\newcommand{\Gm}{\mathbb{G}_m}

\newcommand{\F}{\mathbb{F}}
\newcommand{\srk}{\mathrm{srk}}

\newcommand{\below}[2]{\scriptstyle #1 \atop \scriptstyle #2}

\theoremstyle{definition}
\newtheorem{theorem}{Theorem}
\newtheorem{lemma}[theorem]{Lemma}

\newtheorem{prop}[theorem]{Proposition}
\newtheorem{corollary}[theorem]{Corollary}
\newtheorem{definition}[theorem]{Definition}

\newtheorem*{LP}{Linear Program}

\title[Improved Explicit Upper Bounds for the Cap Set Problem]{Improved Explicit Upper Bounds\\for the Cap Set Problem}
\author{Zhi Jiang}
\date{}

\begin{document}

\begin{abstract}
    Ellenberg and Gijswijt gave the best known asymptotic upper bound for the cardinality of subsets of $\F_q^n$ without 3-term arithmetic progressions. We improve this bound by a factor $\sqrt{n}$. In the case $q=3$, we also obtain more
    explicit upper bounds for the Cap Set Problem.
\end{abstract}

\thanks{The author was partially supported by NSF grant IIS 1837985.}

\maketitle

\section{Introduction}

\subsection{Cap Set Problem}
Let $\mathbb F_q$ be the finite field containing $q$ elements. A {\em cap set} $S$ is a subset of an $n$-dimensional vector space $\mathbb F^n_3$ that does not have an arithmetic progression of length 3, in other words, $S$ does not have three colinear points. It is natural to ask for the largest possible size $c(n)$ of a cap set in $\F_3^n$.
This question is known as the Cap Set Problem. 
For small $n$ we have $c(1)=2$, $c(2) = 4$ and $c(3) = 9$.

We have $2^n \le c(n) \le 3^n$ trivially, a lower bound $\omega(2.217^n)$ was given by Edel \cite{Edel} in 2004, and Tyrrell recently improved it to $\omega(2.218^n)$ \cite{Tyrrell}. Here we are interested in finding an upper bound for this function. It was first shown by Brown and Buhler \cite{BrownBuhler} that $c(n) = o(3^n)$ and this bound was improved to $O(3^n/n)$ by Meshulam \cite{Meshulam}. In 2012, Bateman and Katz \cite{BatemanKatz} lowered the upper bound to $O(3^n/n^{1+\epsilon})$ for some $\epsilon >0$. The next breakthrough was made by Ellenberg and Gijswijt \cite{EllenbergGijswijt}
who showed that $c(n) = O(\theta^n)$ using the polynomial method of Croot, Lev and Pach \cite{CrootLevPach}, where $r=\frac{\sqrt{33}-1}{8}$ and $\theta=\frac{1+r+r^2}{r^{2/3}}\approx 2.7551$.
Tao reformulated this result by using the notion of slice rank~\cite{Tao}. We say a tensor $v\in V = V_1\otimes\cdots\otimes V_d$ has slice rank 1 if it is contained in 
\[
V_1\otimes\cdots\otimes V_{i-1}\otimes w\otimes V_{i+1}\otimes\cdots\otimes V_d
\]
for some $i$ and $w\in V_i$. The slice rank $\srk(v)$ of an arbitrary tensor $v\in V$ is the minimal number $r$ such that $v$ can be written as the sum of $r$ tensors of slice rank 1. 
Let $u$ be the tensor 
\[
u = \sum_{\below{i,j,k\in \mathbb F_3}{ i+j+k = 0}} e_i\otimes e_j\otimes e_k\in \mathbb F^{3\times 3\times 3}_3,
\]
where $\{e_0, e_1, e_2\}$ is a basis of $\mathbb F_3^3$. If $v\in V_1\otimes V_2\otimes\cdots\otimes V_d, w\in W_1\otimes W_2\otimes\cdots\otimes W_d$, then the vertical tensor product (or {\em Kronecker product}) $v\boxtimes w$ is the usual tensor product $v\otimes w$ but viewed as
\[
v\boxtimes w \in (V_1\otimes W_1)\otimes (V_2\otimes W_2)\otimes\cdots\otimes(V_d\otimes W_d).
\]
Tao's idea is to show $c(n)\le \srk(u^{\boxtimes n})$ and compute an upper bound of $\srk(u^{\boxtimes n})$, where
$u^{\boxtimes n}=\underbrace{u\boxtimes u \boxtimes \cdots\boxtimes u}_n$. 

\subsection{Subset of $\mathbb F^n_q$ With No Three-Term Arithmetic Progression}
Let $\mathbb F_q$ be a finite field. It is also interesting to look at a more general problem, which is to find the largest size of a subset in $\mathbb F^n_q$ with no three-term arithmetic progression. If $q=3$, this is the Cap Set Problem.

Suppose $n$ be a positive integer and let $M_n$ be the set of monomials in $x_1,\cdots,x_n$ whose degree in each variable is at most $q-1$. For any $0<d<2n$, let $m_d$ be the number of monomials in $M_n$ of degree at most $d$. An upper bound for the largest size of a subset in $\F_q^n$ with no three-term arithmetic progression was proved by Ellenberg and Gijswijt \cite{EllenbergGijswijt}, the main result is as the following:

\begin{comment}In \cite{EllenbergGijswijt}, Ellenberg and Gijswijt gave an upper bound for the size of a subset in $\mathbb F^n_q$ with no three-term arithmetic progression. The following theorem is the main result of \cite{EllenbergGijswijt}.
\end{comment}

\begin{theorem}[Theorem 4 in \cite{EllenbergGijswijt}]\label{A}
Let $\alpha, \beta, \gamma$ be elements of $\mathbb F_q$ such that $\alpha+\beta+\gamma=0$ and $\gamma \ne 0$, and let $A$ be a subset of $\mathbb F^n_q$ such that all solutions $(a_1,a_2,a_3)\in A^3$ of the equation 
\[
\alpha a_1+\beta a_2+\gamma a_3=0
\]
satisfy $a_1=a_2=a_3$. Then we have \[
|A|\le 3m_{(q-1)n/3}.
\]
\end{theorem}
It was pointed out in \cite{EllenbergGijswijt} that $m_{(q-1)n/3}=O(\theta_q^n)$, where $\theta_q<q$
is the minimal value of $f(x)=\frac{1+x+x^2+\cdots+x^{q-1}}{x^{(q-1)/3}}$ for $x>0$. For example when $q=3$, we get $|A| = O(\theta_3^n)$ where $\theta_3<2.7552$. As a corollary, Ellenberg and Gijswijt obtained the same  upper bound for the Cap Set Problem.

A related problem is the cardinality of tri-colored sum-free sets. A tri-colored sum-free set in $\F_q^n$
is a subset $\{(a_1,b_1,c_1),(a_2,b_2,c_2),\cdots,(a_N,b_N,c_N)\}\subseteq \F_q^n\times \F_q^n\times \F_q^n$ with $N$ elements with the property that $a_i+b_j+c_k=0$ if and only if $i=j=k$. The upper bound of Ellenberg and Gijswijt also works for the tri-colored sum-free sets
so one gets $N=O(\theta_q^n)$. 
Kleinberg, Sawin and Speyer \cite{KSS} showed
that there exists tri-colored sum-free sets with
cardinality $\theta_q^n e^{-2\sqrt{(2\log 2\log\theta_q)n}-O_q(\log n)}$.

\subsection{Main Results of This Paper}
We find improved upper bounds for the Cap Set Problem. Furthermore, we also give the explicit coefficients of the bounds, and it turns out that the coefficients of upper bounds we get depend on $n$ mod 3, more precisely:
\begin{theorem}\label{capset}

For $n \gg 0$, the size of largest possible cap set in $\mathbb F^n_3$ is bounded by
 
 \begin{enumerate}
     \item If $n = 3s$ for some integer $s > 0$, then 
     \[
         c(n) \le 2.4951\frac{\theta^n}{\sqrt{n}}(1+o(1)) = O\left(\frac{\theta^n}{\sqrt{n}}\right).
     \]
     \item If $n = 3s - 1$ for some integer $s > 0$, then
     \[
         c(n) \le 1.7529\frac{\theta^n}{\sqrt{n}}(1+o(1)) = O\left(\frac{\theta^n}{\sqrt{n}}\right).
     \]
     \item If $n = 3s - 2$ for some integer $s > 0$, then
     \[
         c(n) \le 1.2288\frac{\theta^n}{\sqrt{n}}\left(1+o(1)\right) = O\left(\frac{\theta^n}{\sqrt{n}}\right),
     \]
     where $r = \frac{\sqrt{33}-1}{8}$, and $\theta=\theta_3 = \frac{1+r+r^2}{r^{2/3}}\approx 2.7551$.
 \end{enumerate}
 
\end{theorem}

Based on Theorem \ref{A}, we give the upper bound  $|A| = O\big(\frac{\theta_q^n}{\sqrt{n}}\big)$ for some $\theta_q<q$,  which improves the bound of Ellenberg and Gijswijt's  by a factor of $\sqrt n$.
\begin{theorem}\label{ubA}
Let $A$ and $m_d$ as in Theorem \ref{A}, let $f(x) = \frac{1+x+x^2+\cdots+x^{q-1}}{x^{(q-1)/3}}$, and let $0<r<1$ be a positive integer that minimizes $f(x)$ on the positive real axis, then 
\[
|A|\le \frac{3}{(1-r)r}\sqrt{\frac{f(r)}{f''(r)}}\frac{f(r)^n}{\sqrt{2\pi n}}(1+o(1)) = O\left(\frac{f(r)^n}{\sqrt{n}}\right).
\]
In particular, $|A|=O(\frac{\theta_q^n}{\sqrt{n}})$, where $\theta_q=f(r)$.
\end{theorem}

The next section gives some preliminaries for this paper. 

\section{Preliminaries}

\subsection{The $G$-Stable Rank For Tensors}
In \cite{Derksen}, Harm Derksen introduced $G$-stable rank for tensors. Here we give a brief introduction of $G$-stable rank for tensors. We refer to the original paper \cite{Derksen} for more details. Suppose the base field $K$ is perfect.
Let $\Gm$ be the multiplicative group over $K$.
A 1-parameter subgroup of an algebraic group $G$ is a homomorphism of algebraic groups $\lambda:{\mathbb G}_m\to G$. If $\lambda:\Gm\to \GL_n$ is a $1$-parameter subgroup,
then we can view $\lambda(t)$ as an invertible $n\times n$ matrix whose entries lie in the ring $K[t,t^{-1}]$ of Laurent polynomials. We say that $\lambda(t)$ is a polynomial $1$-parameter subgroup of $\GL_n$ if all these entries lie in the polynomial ring $K[t]$. 
 Consider the action of the group $G = \GL(V_1)\times \GL(V_2)\times\cdots\times \GL(V_d)$ on the tensor product space $V = V_1\otimes V_2\otimes\cdots\otimes V_d$. 
A 1-parameter subgroup $\lambda:\Gm\to G$ 
can be written as
\[
\lambda(t) = (\lambda_1(t), \cdots, \lambda_d(t))
\]
where $\lambda_i(t)$ is a $1$-parameter subgroup of $\GL(V_i)$ for all $i$. We say that $\lambda(t)$ is polynomial if and only if $\lambda_i(t)$ is a polynomial $1$-parameter subgroup for all $i$.

The $t$-valuation $\val(a(t))$ of a polynomial $a(t)\in K[t]$ is the biggest integer $n$ such that $a(t) = t^nb(n)$ for some $b(t)\in K[t]$. For $a(t), b(t)\in K[t]$, the $t$-valuation $\val\big(\frac{a(t)}{b(t)}\big)$ of the rational function $\frac{a(t)}{b(t)}\in K(t)$ is $\val\big(\frac{a(t)}{b(t)}\big) = \val(a(t)) - \val(b(t))$. For a tuple $u(t) = (a_1(t),a_2(t),\cdots, a_d(t))\in K(t)^d$, we define the $t$-valuation of $u(t)$ as
\begin{equation}\label{val}
\val(u(t)) = \min_i\{\val(a_i(t))| 1\le i\le d\}.
\end{equation}

If $\lambda$ is a $1$-parameter subgroup of $G$ and $v\in V$ is a tensor, then we have $\lambda(t)\cdot v\in K(t)\otimes V$. We view $K(t)\otimes V$ as a vector space over $K(t)$ and define the $t$-valuation $\val(\lambda(t)\cdot v)$ as in \eqref{val}. Assume $\val(\lambda(t)\cdot v) > 0$, then for any $\alpha = (\alpha_1,\alpha_2,\cdots,\alpha_d)\in \mathbb R^d_{>0}$, we define the slope
\begin{equation}
    \mu_\alpha(\lambda(t), v) = \frac{\sum_{i = 1}^d\alpha_i \val(\det(\lambda_i(t)))}{\val(\lambda(t)\cdot v)}.
\end{equation}

The $G$-stable rank for $v\in V$ is the infimum of the slope with respect to all such polynomial 1-parameter subgroups. More precisely:

\begin{definition}(\cite{Derksen} Theorem 2.4)
If $\alpha\in \mathbb R^d_{>0}$, then the $G$-stable rank $\rk^G_\alpha(v)$ is the infimum of $\mu_\alpha(\lambda(t), v)$ where $\lambda(t)$ is a polynomial 1-parameter subgroup of $G$ and $\val(\lambda(t)\cdot v) > 0$. If $\alpha = (1,1,\cdots, 1)$, we simply write $\rk^G(v)$.
\end{definition}

The $G$-stable rank is used to give an upper bound for the cap set as shown in \cite{Derksen}. Let $K = \mathbb F_3$, we view $K^{3^n}$ as the vector space with basis $[a], a\in \mathbb F_3^n$. Consider the tensor

\[
v = \sum_{\below{(a,b,c)\in \mathbb F^{n\times 3}_3}{ a+b+c = 0}}[a]\otimes [b]\otimes [c] = \sum_{\below{(a,b,c)\in \mathbb F^{n\times 3}_3} {a+b+c = 0}}[a,b,c]\in K^{3^n}\otimes K^{3^n}\otimes K^{3^n}.
\]
Let $S\subset \mathbb F^n_3$ be a cap set, and we project $v$ onto the subset $S^3\subset \mathbb F^{n\times 3}_3$, we get 
\[
w = \sum_{\below{(a,b,c)\in S^3} {a+b+c=0}}[a,b,c] = \sum_{a\in S}[a,a,a],
\]
here we used the fact that $a+b+c = 0$ in $\mathbb F_3^n$ with $a, b, c \in S$ if and only if $a = b = c$.  It was shown in \cite{Derksen} that 
\begin{theorem}
The size of cap set is bounded by the $G$-stable rank of $w$ and $v$, i.e.
\[
|S| \le \rk^G(w) \le \rk^G(v).
\]
Furthermore, the $G$-stable rank of $v$ is bounded by
\[
\rk^G(v)\le 3\sum_{i = 0}^{2n} f_{n,i}t_i,
\]
where $f_{n,i}$ is the coefficient of $x^i$ in $(1+x+x^2)^n$, and $t_0,t_1,\cdots,t_{2n} \ge 0$ are numbers such that $t_i+t_j+t_k\ge 1$ whenever $i+j+k\le 2n$. 
\end{theorem}

Guided by the above theorem, we have the following linear program:
\begin{LP}\label{L}Let $f_{n,i}$ be the coefficient of $x^i$ in the polynomial $(1+x+x^2)^n$, minimize the summation 
\[
3\sum_{i=0}^{2n}f_{n,i}t_i
\]
under the following constraints:

\begin{enumerate}
\item $t_i+t_j+t_k\ge 1$ if $i+j+k\le 2n$;
\item $t_i\ge 0$ for all $i$.
\end{enumerate}

\end{LP}

We would like to find an optimal solution to the Linear Program.
A conjecture of optimal solution for the Linear Program was made in \cite{Derksen}, we will prove this conjecture, at least for large $n$, and this will give an upper bound for the $G$-stable rank of the tensor $v$. It turns out that the optimal solution depends on $n$ mod 3:

\begin{theorem}[\cite{Derksen}, Conjecture 6.1]\label{optimum}
An optimal solution $(t_0,t_1,\dots,t_{2n})$ of the Linear Program is given by:

\begin{enumerate}
\item If $n=3s$, $t_i=1$ for $0\le i\le 2s-2$, $t_{2s-1}=\frac{2}{3}, t_{2s}=\frac{1}{3}$, and $t_i=0$ for $i\ge 2s+1$.

 \item If $n=3s-1$, $t_i=1$ for $0\le i\le 2s-4$, $t_{2s-3}=\frac{4}{5}, t_{2s-2}=\frac{3}{5},t_{2s-1}=\frac{2}{5}, t_{2s}=\frac{1}{5}$, and $t_i=0$ for $i\ge 2s+1$.
 
 \item If $n=3s-2$, $t_i=1$ for $0\le i\le 2s-4$, $t_{2s-3}=\frac{3}{4}, t_{2s-2}=\frac{2}{4},t_{2s-1}=\frac{1}{4},$ and $t_i=0$ for $i\ge 2s$.
 \end{enumerate}
 
\end{theorem}

\subsection{Estimation of Coefficients}
 Let $q>0$ be a positive integer, and we denote the coefficient of $x^i$ in the polynomial $(1+x+x^2+\cdots+x^{q-1})^n$ by $f_{n,i}$.

\begin{theorem} \label{gestimate}
Let us fix $\alpha$ with $0<\alpha<\frac{q-1}{2}$ and $B>0$.  Let $f(x)=\frac{1+x+x^2+\cdots+x^{q-1}}{x^\alpha}$ and $r$ is a positive number such $f(r)$ is minimal along the real positive axis. Then as $n\to \infty$ we have the following asymptotic behaviors for all $\beta$ with $|\beta|<B$ and $\alpha n+\beta\in \Z$:

\begin{equation}\label{gf}
     f_{n,\alpha n + \beta} =  \frac{f(r)^n}{\sqrt{2\pi n}}\frac{1}{r^{\beta+1}}\sqrt{\frac{f(r)}{f''(r)}}(1 + o(1)).
\end{equation}
If $0<r<1$, we also have
\begin{equation}\label{gsf}
     \sum_{k=0}^{\alpha n + \beta}f_{n,k} = \frac{f(r)^n}{\sqrt{2\pi n}}\frac{1}{(1-r)r^{\beta+1}}\sqrt{\frac{f(r)}{f''(r)}}(1 + o(1)).
\end{equation}
\end{theorem}

By the above estimation of coefficients, we can give an asymptotic behavior of $m_{(q-1)n/3}$ and therefore give a proof of Theorem \ref{ubA}. Recall that for any $0<d<2n$, $m_d$ is the number of monomials in $x_1,\cdots,x_n$ with total degree at most $d$ and in which each variable appears with degree at most $q-1$. 

\textbf{Proof of Theorem \ref{ubA}:}
Let $f_{n,i}$ be the coefficient of $x^i$ in the polynomial $(1+x+x^2+\cdots+x^{q-1})^n$, we can write 
\[
m_{(q-1)n/3} = \sum_{i=0}^{(q-1)n/3}f_{n,i}.
\]
Then by equation \eqref{gsf} of Theorem \ref{gestimate}, as long as we can find some $0<r<1$ that minimizes $f(x) = \frac{1+x+x^2+\cdots+x^{q-1}}{x^{(q-1)/3}}$ on the positive real axis, we have 
\[
m_{(q-1)n/3} = \frac{f(r)^n}{\sqrt{2\pi n}}\frac{1}{(1-r)r}\sqrt{\frac{f(r)}{f''(r)}}(1 + o(1)).
\]
Indeed, such $r$ exists. Let $\alpha = \frac{q-1}{3}$, then 
\[
f'(1) = (q-1-\alpha) + (q-2-\alpha) + \cdots+(1-\alpha)-\alpha
\]
\[
=\frac{(q-1-2\alpha)q}{2}>0.
\]
However, $\lim_{x\to 0^+}f(x)=+\infty$, so there must be some $0<r<1$ that minimizes $f(x)$. The upper bound of $|A|$ in Theorem \ref{ubA} follows immediately. This completes the proof of Theorem \ref{ubA}.

To solve the Linear Program, we need a good estimation of the coefficients $f_{n,i}$ of $x^i$ in $(1+x+x^2)^n$. This is the case of Theorem \ref{gestimate} when $q=3$, we state the result as a corollary:
\begin{corollary}\label{estimate}
 Let us fix $\alpha$ with $0<\alpha<1$ and $B>0$.  Let $f(x)=\frac{1+x+x^2}{x^\alpha}$ and $r$ is a positive number such $f(r)$ is minimal along the real positive axis. Then as $n\to \infty$ we have the following asymptotic behaviors for all $\beta$ with $|\beta|<B$ and $\alpha n+\beta\in \Z$:
 
\begin{equation}\label{f}
     f_{n,\alpha n + \beta} = \frac{f(r)^n}{\sqrt{2\pi n}}\frac{1}{r^\beta}\sqrt{\frac{1+r+r^2}{2\alpha-(1-\alpha)r}}(1 + o(1)).
\end{equation}
If $0<r<1$, we also have
\begin{equation}\label{sf}
     \sum_{k=0}^{\alpha n + \beta}f_{n,k}= \frac{f(r)^n}{\sqrt{2\pi n}}\frac{1}{(1-r)r^\beta}\sqrt{\frac{1+r+r^2}{2\alpha-(1-\alpha)r}}(1 + o(1)).
\end{equation}
\end{corollary}

We will prove Theorem \ref{gestimate} in next section via the Residue Theorem and give some interesting inequalities which will be used in the proof of Theorem \ref{optimum}, and after that we will proceed to the proof of Theorem \ref{optimum}.

\section{Estimation of Coefficients}

\subsection{Express $f_{n,i}$ Via Residue Theorem}
This section devotes to the proof of Theorem~\ref{gestimate}. Let $q>0$ be a positive integer, the goal is to estimate the coefficient $f_{n,i}$ of $x^{i}$ in the expansion of $(1+x+x^2+\cdots+x^{q-1})^n$ for some large positive integer $n$. Let $i = \alpha n+\beta$ for some $0 < \alpha < \frac{q-1}{2}$ and $\beta$ whose absolute value $|\beta|$ is bounded. 

Let us define $f(z)=\frac{(1+z+z^2+\cdots + z^{q-1})}{z^{\alpha}}=z^{-\alpha}+z^{-\alpha+1}+z^{-\alpha+2}+\cdots+z^{-\alpha+q-1}$. By the Residue Theorem, the coefficient $f_{n,\alpha n+\beta}$ is equal to the following integral

\begin{equation}\label{integral}
f_{n,\alpha n+\beta}=\frac{1}{2\pi i}\oint \frac{(1+z+z^2+\cdots + z^{q-1})^n}{z^{\alpha n+\beta}}\frac{dz}{z}=\frac{1}{2\pi i}\oint f(z)^n\frac{dz}{z^{1+\beta}}.
\end{equation}

The integral is taken over a circle centered at origin, which is independent of the radius.
We analyze the absolute value of $f(z)$ on the circle of radius $r$. If $z=r e^{it}$, then we have

\[
|f(z)|=\frac{|1+z+\cdots + z^{q-1}|}{r^\alpha}
\le \frac{1+|z|+\cdots + |z|^{q-1}}{r^\alpha}
=\frac{1+r+\cdots + r^{q-1}}{r^\alpha}=f(r).
\]
For fixed $r$, $|f(z)|$  has maximal value at $t=0$. We let $r>0$ be a positive number such that $f(r)$ is minimal along the real positive axis. Indeed, we have
\[
f'(r)=-\alpha r^{-\alpha-1}+(1-\alpha)r^{-\alpha}+(2-\alpha)r^{-\alpha+1}+\cdots+(q-1-\alpha)r^{-\alpha+q-2}
\]
\[
=r^{-\alpha-1}((q-1-\alpha)r^{q-1}+ \cdots+(2-\alpha)r^2+(1-\alpha)r-\alpha)).
\]

So $r$ satisfies

\begin{equation}\label{radius}
(q-1-\alpha)r^{q-1}+ \cdots+(2-\alpha)r^2+(1-\alpha)r-\alpha=0.
\end{equation}

We will compute the integral  over the circle centered at origin with radius $r$. Since $z = r e^{it}$ and $f(z)$ has maximum magnitude at $t = 0$, we can expand $f(z)^n$ near $t=0$ as a Taylor series. Let $g(t)=f(r e^{it})$, we take the derivatives of $g(t)$ with respect to $t$: 

\[
g'(t)=ir e^{it}f'(r e^{it}),
\]
\[
g''(t)=-r e^{it}f'(r e^{it})-r^2 e^{2it}f''(r e^{it}).
\]

Using the fact that $f'(r) = 0 $ and $g'(0) = ir f'(r)=0$, we have

\[
g(t)=g(0)+g'(0)t+{\textstyle \frac{1}{2}}g''(0)t^2+ O(t^3)
\]
\[
=f(r)-{\textstyle \frac{1}{2}}r^2f''(r)t^2+ O(t^3).
\]

It is easy to see that $f''(r)=r^{-\alpha-2}[-\alpha(-\alpha-1)-\alpha(1-\alpha)r+\cdots+(q-1-\alpha)(q-2-\alpha)r^{q-1}]$. Since $f(r)$ has minimal value at $r$ along the real axis, we have $f''(r) > 0$. Let $\gamma = \frac{r^2f''(r)}{f(r)} > 0$, we get

\begin{equation}
g(t)=
%f(r)\left[1-\frac{r^2}{2}\frac{f''(r)}{f(r)}t^2\right] + O(t^3) =
f(r)\big(1 - \frac{\gamma}{2}t^2\big) + O(t^3)
=f(r)e^{-\frac{\gamma t^2}{2}}+ O(t^3).
\end{equation}

For large $n$, we want to compute
\begin{equation}
f_{n,\alpha n + \beta} = \frac{1}{2\pi r^\beta} \int_{-\pi}^\pi g(t)^ne^{-i\beta t}dt.
\end{equation}

Let $B>0$ and $C > 0$ be some small constants to be determined later. We can split the integral as

\begin{equation}
f_{n,\alpha n + \beta} = \frac{f(r)^n}{2\pi r^\beta \sqrt{\gamma n}} \sqrt{\gamma n} \Big(\int_{B}^\pi \Big(\frac{g(t)}{f(r)}\Big)^ne^{-i\beta t}dt + \int_{\mathsmaller{{C\sqrt{\frac{\log n}{n}}}}}^B \Big(\frac{g(t)}{f(r)}\Big)^ne^{-i\beta t}dt 
\end{equation}
\[
+\int_{-\mathsmaller{C\sqrt{\frac{\log n}{n}}}}^{C\sqrt{\frac{\log n}{n}}} \Big(\frac{g(t)}{f(r)}\Big)^ne^{-i\beta t}dt + \int_{-B}^{-C\sqrt{\frac{\log n}{n}}} \Big(\frac{g(t)}{f(r)}\Big)^ne^{-i\beta t}dt + \int_{-\pi}^{-B} \Big(\frac{g(t)}{f(r)}\Big)^ne^{-i\beta t}dt\Big).
\]

Since $g(t) = f(re^{it})$ is continuous and $|\frac{g(t)}{f(r)}|<1$ when $B \leq t \leq \pi$,
there exists a constant $\delta<1$ such that $|\frac{g(t)}{f(r)}|\leq \delta$
for $t\in [B,\pi]$.
Therefore we have
\begin{equation}\label{r1}
\sqrt{\gamma n}\int_{B}^\pi \Big(\frac{g(t)}{f(r)}\Big)^ne^{-i\beta t}dt \le \sqrt{\gamma n} \int_{B}^\pi \Big|\frac{g(t)}{f(r)}\Big|^ndt  \leq \sqrt{\gamma n}\pi\delta^n\to 0\quad \text{as}\quad n\to \infty.
\end{equation}
Similarly, we get
\begin{equation}\label{r2}
\sqrt{\gamma n}\int^{-B}_{-\pi} \Big(\frac{g(t)}{f(r)}\Big)^ne^{-i\beta t}dt \le \sqrt{\gamma n}\int^{-B}_{-\pi} \Big|\frac{g(t)}{f(r)}\Big|^ndt \leq \sqrt{\gamma n}\pi\delta^n\to 0\quad \text{as}\quad n\to \infty.
\end{equation}
\begin{lemma}
If $h(z)\in \mathbb C[z,z^{-1}]$ is a Laurent polynomial, then the function $\frac{d}{d\,t}|h(e^{it})|^2$ has only finitely many zeros for $t\in [0,2\pi]$.
\end{lemma}
\begin{proof}
The functions $|h(e^{it})|^2=h(e^{it})\overline{h}(e^{-it})$ and $\frac{d}{d\,t}|h(e^{it})|^2$ are Laurent polynomials in $e^{it}$, so $\frac{d}{d\,t}|h(e^{it})|^2=0$ has only finitely many solutions for $e^{it}$.
\end{proof}

\begin{lemma}\label{BC}
Over the interval $\big[C{\scriptstyle\sqrt{\frac{\log n}{n}}}, B\big]$, for sufficient large $n$ and sufficient small $B > 0$, $\big|\frac{g(t)}{f(r)}e^{-i\frac{\beta}{n} t}\big|=
\big|\frac{g(t)}{f(r)}\big|$ is largest at $t = C{\scriptstyle\sqrt{\frac{\log n}{n}}}$.
\end{lemma}
\begin{proof}
 Since $|g(t)|=|f(re^{it})|=|1+re^{it}+r^2e^{2it}+\cdots+r^{q-1}e^{(q-1)it}|$, $\frac{d}{d\,t}|g(t)|^2=2|g(t)|\frac{d}{d\,t}|g(t)|$ and $\frac{d}{d\,t}|g(t)|$ only have finitely many zeros in the interval $[0,2\pi]$.
So for a small $B>0$,
the function $|g(t)|$ is monotone on the interval $(0,B)$.
Recall that $\big|\frac{g(t)}{f(r)}e^{-i\frac{\beta}{n} t}\big| = \big|\frac{g(t)}{f(r)}\big| \le 1$ and it takes maximal value 1 at $t = 0$.
So $\big|\frac{g(t)}{f(r)}\big|$ is decreasing on the interval $(0,B)$ and on the interval $\big[C{\scriptstyle\sqrt{\frac{\log n}{n}}},B\big]$ it has a maximum at $t=C{\scriptstyle\sqrt{\frac{\log n}{n}}}$.

%As $n$ is sufficiently large, $C\sqrt{\frac{\log n}{n}}$ is arbitrarily close to $0$, we can pick $B > 0$ sufficiently small such that the above condition holds, then over the interval $[C\sqrt{\frac{\log n}{n}}, B]$, $\Big|\Big(\frac{g(t)}{f(r)}e^{-i\frac{\beta}{n} t}\Big)^{n}\Big|$ takes largest value at $t = C\sqrt{\frac{\log n}{n}}$.
\end{proof}

Let us fix $B > 0$ as in Lemma \ref{BC}, then we have 
\begin{equation}\label{eq3}
\Big|
\sqrt{\gamma n}\int_{C\sqrt{\frac{\log n}{n}}}^B \Big(\frac{g(t)}{f(r)}e^{-i\frac{\beta}{n} t}\Big)^{n}dt\Big| \le \sqrt{\gamma n}B\left[\ \Big|\frac{g(t)}{f(r)}\Big|^n\ \right]_{t = C\sqrt{\frac{\log n}{n}} }.
\end{equation}
 % as $n\to \infty$, $C\sqrt{\frac{log n}{n}}\to 0$, we have
%\begin{equation}
%\lim_{n\to \infty} \sqrt{\gamma n}B\Big|\Big(\frac{g(t)}{f(r)}\Big)^n|_{t = C\sqrt{\frac{\log n}{n}}} \Big|= \lim_{n\to \infty} \sqrt{\gamma n} B \Big|\Big(\frac{g(t)}{f(r)}\Big)^n|_{t = C\sqrt{\frac{log n}{n}}} \Big|
%\end{equation}
Recall the expansion of $\frac{g(t)}{f(r)}$ in a neighborhood of $t=0$:
\begin{equation}
\frac{g(t)}{f(r)} = 1 - \frac{\gamma}{2}t^2 + O(t^3),
\end{equation}
where $\gamma = \frac{r^2 f''(r)}{f(r)} > 0$. Then we have the following limit behavior of (\ref{eq3}) when $n\to \infty:$
\begin{multline}
\lim_{n\to\infty} \sqrt{\gamma n}B\left[ \Big|\frac{g(t)}{f(r)}\Big|^n\ \right]_{t = C\sqrt{\frac{\log n}{n}} }=\lim_{n\to\infty}
 \sqrt{\gamma n}B \left(1-\mathsmaller{\frac{\gamma C^2\log n}{2n}}+O\Big(\big(\mathsmaller{\frac{\log n}{n}}\big)^{3/2}\Big)\right)^n\\
%\lim_{n\to \infty} B\sqrt{\gamma n}\Big |\Big(\frac{g(t)}{f(r)}\Big)^n|_{t = C\sqrt{\frac{log n}{n}}} \Big |
=\lim_{n\to\infty} \sqrt{\gamma n}B\left (\Big(1 - \textstyle{\frac{\gamma C^2\log n}{2 n}} + O\Big(\big({\textstyle\frac{\log n}{n}}\big)^{3/2}\Big)\Big)^{\frac{2n}{\gamma C^2\log n}} \right)^{\frac{\gamma C^2\log n}{2}}\\
=\lim_{n\to\infty}\sqrt{\gamma n} B e^{-\frac{\gamma C^2\log n}{2}} =\lim_{n\to\infty} B\sqrt{\gamma} e^{-\frac{(\gamma C^2 - 1)\log n }{2}}.
\end{multline}
We fix any $C > 0$ such that  $C^2\gamma - 1 > 0$. As a result, we get
$$
\lim_{n\to\infty}
 \sqrt{\gamma n}B\left[\ \Big|\frac{g(t)}{f(r)}\Big|^n\ \right]_{t = C\sqrt{\frac{\log n}{n}} }=
\lim_{n\to \infty} B\sqrt{\gamma}e^{-\frac{(\gamma C^2 - 1)\log n}{2}}= 0.$$
From (\ref{eq3}) it follows that 
\begin{equation}\label{r3}
\lim_{n\to\infty}\sqrt{\gamma n}\int_{C\sqrt{\frac{\log n}{n}}}^B \Big(\frac{g(t)}{f(r)}e^{-i\frac{\beta}{n} t}\Big)^{n}dt = 0.
\end{equation}

Next we compute 

\begin{equation}\label{eq4}
\sqrt{\gamma n}\int_{-C\sqrt{\frac{\log n}{n}}}^{C\sqrt{\frac{\log n}{n}}} \Big(\frac{g(t)}{f(r)}\Big)^ne^{-i\beta t}dt = \int^{C\sqrt{\gamma \log n}}_{-C\sqrt{\gamma \log n}}\left(\frac{g(\frac{s}{\sqrt{\gamma n}})}{f(r)}\right)^n e^{\frac{-i\beta s}{\sqrt{\gamma n}}}ds,
\end{equation}
where $s = t\sqrt{\gamma n} $. For any $t$ in the interval  $[-C{\scriptstyle\sqrt{\frac{\log n}{n}}},C{\scriptstyle\sqrt{\frac{\log n}{n}}}]$, we have
\begin{equation}\frac{g(t)}{f(r)} = 1-\frac{\gamma}{2}t^2 + O(t^3) =e^{-\frac{\gamma}{2}t^2}+O(t^3)= e^{-\frac{\gamma}{2}t^2} + O\big(({\textstyle\frac{\log n}{n}})^{\frac{3}{2}}\big).
\end{equation}
From this it follows that
\begin{equation}
\left(\frac{g(\frac{s}{\sqrt{\gamma n}})}{f(r)}\right)^n = \left(e^{-\frac{\gamma}{2}(\frac{s}{\sqrt{\gamma n}})^2} +
O\big(({\textstyle\frac{\log n}{n}})^{\frac{3}{2}}\big)\right)^n = e^{-\frac{s^2}{2}} + O\big({\textstyle \frac{(\log n)^{\frac{3}{2}}}{\sqrt n}}\big).
\end{equation}
Therefore integral (\ref{eq4}) becomes
\begin{equation}\label{eq5}
\int^{C\sqrt{\gamma \log n}}_{-C\sqrt{\gamma \log n}}\left(e^{-\frac{s^2}{2}} + O\big({\textstyle\frac{(\log n)^{\frac{3}{2}}}{\sqrt n}}\big)\right)
e^{\frac{-i\beta s}{\sqrt{\gamma n}}}ds.
\end{equation}
We want to find the behavior of the integral when $n\to \infty$, in this case the factor $e^{\frac{-i\beta s}{\sqrt{\gamma n}}}\to 1$ and does not contribute to the integral, so (\ref{eq5}) becomes
\begin{equation}\label{r4}
    \lim_{n\to\infty} \int^{C\sqrt{\gamma \log n}}_{-C\sqrt{\gamma \log n}}\Big(e^{-\frac{s^2}{2}} + O\big({\textstyle \frac{(\log n)^{\frac{3}{2}}}{\sqrt n}}\big)\Big)ds = \lim_{n\to\infty}\Big(\int^\infty_{-\infty} e^{-\frac{s^2}{2}}ds + O\big({\textstyle \frac{(\log n)^2}{\sqrt n}}\big)\Big) = \sqrt{2\pi}.
\end{equation}
Finally by (\ref{r1}), (\ref{r2}), (\ref{r3}) and (\ref{r4}), we get
\begin{equation}
    f_{n,\alpha n + \beta} = \frac{f(r)^n}{r^{\beta}\sqrt{2\pi\gamma n}}\big(1 + o(1)\big) = \frac{f(r)^n}{\sqrt{2\pi n}}\frac{1}{r^{\beta+1}}\sqrt{\frac{f(r)}{f''(r)}}\big(1 + o(1)\big).
\end{equation}

%So we have an estimation of coefficient of $x^{\alpha n+\beta}$ in the polynomial $(1 + x+ x^2 + \cdots + x^{q - 1})^n$. 

This proves the formula (\ref{gf}) in Theorem \ref{gestimate}. It remains to prove formula (\ref{gsf}) in Theorem \ref{gestimate}. Again by the Residue Theorem, we can write the summation $\sum_{i=0}^{\alpha n + \beta}f_{n,i}$ as

\begin{comment}
 If $q>0$ is a positive integer and $f_{n,i}$ is the coefficient of $x^i$ in $(1+x+x^2+\cdots+x^{q-1})^n$ and assume $i = \alpha n + \beta \in \mathbb Z$ for some $0<\alpha<\frac{q-1}{2}$ and $\beta$ whose absolute value $|\beta|$ is bounded. Let $r$ be a positive number that minimizes $\frac{1+x+x^2+\cdots+x^{q-1}}{x^{\alpha}}$ at real positive axis such that $0<r<1$, then 
 \begin{equation}
 \sum_{k=0}^{\alpha n + \beta}f_{n,k}=\frac{f(r)^n}{\sqrt{2\pi n}}\frac{1}{(1-r)r^{\beta+1}}\sqrt{\frac{f(r)}{f''(r)}}(1 + o(1))
 \end{equation}
\end{comment}

\begin{equation}\label{sum}
\sum_{i = 0}^{\alpha n + \beta} f_{n,i}
 =\frac{1}{2\pi i}\oint \frac{f(z)^n}{z^\beta}(\sum_{n=0}^\infty z^n)\frac{dz}{z}=\frac{1}{2\pi i}\oint \frac{f(z)^n}{1-z}\frac{dz}{z^{1+\beta}}.
\end{equation}
The term $\frac{1}{1-z}$ on the right hand side of (\ref{sum}) contributes a factor $\frac{1}{1-r}$, so formula (\ref{gsf}) in Theorem \ref{gestimate} follows from the same computation. This completes the proof of Theorem \ref{gestimate}. 

\subsection{Some Useful Inequalities}
In this section, $n$ is a sufficiently large integer.
With the estimation of coefficients $f_{n,i}$ in previous section, we prove some useful inequalities which will be used in the proof of Theorem~\ref{optimum} in next section. We denote $[x]$ the integral part of $x$. First we give an interesting observation.

\begin{lemma}
By symmetry, we have $f_{n,i} = f_{n,2n-i}$.
\end{lemma}

\begin{proof}
Since $f_{n,i}$ is the coefficient of $x^i$ in the expansion of the polynomial $(1+x+x^2)^n$, if we write $(1+x+x^2)^n=x^n(x^{-1}+1+x)^n$, then $f_{n,i}$ is the coefficient of $x^{i-n}$ in the expansion $(x^{-1}+1+x)^n$, and $f_{n,2n-i}$ is the coefficient of $x^{n-i}$ in the expansion of $(x^{-1}+1+x)^n$. By the symmetry of $(x^{-1}+1+x)^n$, we have $f_{n,i} = f_{n,2n-i}$.
\end{proof}

\begin{lemma}\label{1}
$\sum_{i=0}^{ [n-s/2]}f_{n,i}-f_{n,s}<0$ for $n\gg 0$ and all $s>n$.
\end{lemma}

\begin{proof}
Let $s=n+k$, $\alpha n=n-k$, then $n>k>0$, $0< \alpha <1$ and $f_{n,s} = f_{n, 2n-s} = f_{n,n-k} = f_{n,\alpha n}$, the above inequality becomes
\[
\sum_{i=0}^{[\alpha n/2]}f_{n,i}-f_{n,\alpha n}<0.
\]

Let $r_1$ be a positive real number that minimizes $\frac{1+x+x^2}{x^{\alpha/2}}$, let $r_2$ be a positive real number that minimizes $\frac{1+x+x^2}{x^\alpha}$. As $0<\alpha<1$, the minimal value of $\frac{1+x+x^2}{x^\alpha}$ on positive real axis increases as $\alpha$ increases. Therefore we have
\begin{equation}\label{r}
\frac{1+r_1+r_1^2}{r_1^{\alpha/2}} < \frac{1+r_2+r_2^2}{r_2^{\alpha}}.
\end{equation}

By Corollary \ref{estimate}, we have for large $n$
\[
\sum_{i=0}^{[\alpha n/2]}f_{n,i}= \frac{C_1}{\sqrt{n}}\left(\frac{1+r_1+r_1^2}{r_1^{\alpha/2}}\right)^n(1 + o(1))
\]
and 
\[
f_{n,\alpha n}= \frac{C_2}{\sqrt{n}}\left(\frac{1+r_2+r_2^2}{r_2^\alpha}\right)^n(1+o(1))
\]
for some positive constants $C_1$ and $C_2$. Then it is clear that $\sum_{i=0}^{[\alpha n/2]}f_{n,i}-f_{n,\alpha n}<0$ by inequality \eqref{r}.
\end{proof}

Similarly, we can prove the following inequality

\begin{lemma}\label{2}
We have $2f_{n,0}+\sum_{i=1}^{[n/2]}f_{n,i}-f_{n,n}<0$ for $n\gg 0$.
\end{lemma}
\begin{proof}
We estimate the asymptotic behavior of $2f_{n,0}+\sum_{i=1}^{[n/2]}f_{n,i}$, which is dominated by the second term. In this case, $\alpha=\frac{1}{2}$, and by solving the equation $\eqref{radius}$, we have $r=\frac{\sqrt{13}-1}{6}$, and $f(r)< 2.4626$. By Corollary \ref{estimate} we get
\[
2f_{n,0}+\sum_{i=1}^{[n/2]}f_{n,i}= O(2.4626^n).
\]
Similarly by Corollary \ref{estimate}, for $\alpha=\frac{2}{3}$ we get $r= \frac{\sqrt{33}-1}{8}$, $f(r)>2.7551$ and hence $f_{n,\frac{2}{3}n}= \Omega(2.7551^n)$.
%So we get $f_{n,n}>f_{n,\frac{2}{3}n}= \Omega(2.7551^n)$ by the computation in Corollary~\ref{capset}.
Therefore $2f_{n,0}+\sum_{i=1}^{[n/2]}f_{n,i}< f_{n, \frac{2}{3}n} < f_{n,n}$ for $n\gg 0$.
\end{proof}

As we said earlier, an optimal solution to the Linear Program in section \ref{L} depends on $n$ mod $3$, next we will consider the three cases separately.

\begin{lemma}\label{3s}
If $n=3s$, then we have $f_{n,2k-1}+2f_{n,2k}+f_{n,2k+1}+\cdots+f_{n,[(n+k)/2]}-f_{n,n-k}<0$ for $0<k<s$ and $n\gg 0$.

\begin{enumerate}
    \item We first show that  $f_{n,2k-1}+2f_{n,2k}+f_{n,2k+1}+\cdots+f_{n,[(n+k)/2]}-f_{n,n-k}<0$ is true for $0<k<s-2$.

\begin{proof}
We will prove a stronger result
\[
2\sum_{i=0}^{[(n+k)/2]}f_{n,i}-f_{n,n-k}<0.
\]
The largest $k$ such that the above inequality holds is $k=s-3$. It is not hard to see that the strongest inequality among them is when $k = s-3$, so we only need to show the above inequality for $k=s-3$. In this case, the inequality becomes
\[
2\sum_{i=0}^{2s-2}f_{n,i}-f_{n,2s+3}<0.
\]
By formula $\eqref{sf}$ in Corollary \ref{estimate}, we see that $2\sum_{i=0}^{2s-2}f_{n,i} = 2C\frac{f(r)^n}{\sqrt n}\frac{r^2}{1-r}(1+o(1))$, and $f_{n,2s+3}= C\frac{f(r)^n}{\sqrt nr^3}(1+o(1))$, where $C=\frac{1}{2\pi}\sqrt{\frac{1+r+r^2}{2\alpha-(1-\alpha)r}}$ and $\alpha=\frac{2}{3}$, $r=\frac{\sqrt{33}-1}{2}$. Therefore it suffices to show
\[
2\frac{r^2}{1-r}-\frac{1}{r^3}<0.
\]
The left hand side is $-3.0651<0$.
\end{proof}

\item We show $f_{n,2k-1}+2f_{n,2k}+f_{n,2k+1}+\cdots+f_{n,[(n+k)/2]}-f_{n,n-k}<0$ is also true for $k=s-2$ and $k=s-1$.
\begin{proof}
When $k=s-2$, the inequality becomes
\[
f_{n,2s-5}+2f_{n,2s-4}+f_{n,2k-3}+f_{n,2s-2}+f_{n,2s-1}-f_{n,2s+2}<0.
\]
By formula $\eqref{f}$ in Corollary \ref{estimate}, we have $f_{n-i}= C\frac{f(r)^n}{\sqrt n}r^i(1+o(1))$ for small $i$. Therefore it suffices to show that 
\[
r^5+2r^4+r^3+r^2+r-\frac{1}{r^2}<0.
\]
The left hand side is $-1.3689<0$.

When $k=s-1$, the inequality becomes
\[
f_{n,2s-3}+2f_{n,2s-2}+f_{n,2s-1}-f_{n,2s+1}<0.
\]
It suffices to show that
\[
r^3+2r^2+r-\frac{1}{r}<0.
\]
The left hand side is $-0.1810<0$
\end{proof}

\end{enumerate}
\end{lemma}

We have proved all inequalities in the case $n=3s$. Next let us consider the case $n=3s-1$.

\begin{lemma}\label{3s-1}
If $n=3s-1$, then we have $f_{n,2k-1}+2f_{n,2k}+f_{n,2k+1}+\cdots+f_{n,[(n+k)/2]}-f_{n,n-k}<0$ for $0<k<s-1$ and $n \gg 0$.
\end{lemma}

\begin{proof}\ 
\begin{enumerate}
    \item We first show $f_{n,2k-1}+2f_{n,2k}+f_{n,2k+1}+\cdots+f_{n,[(n+k)/2]}-f_{n,n-k}<0$ is true when $k<s-2$. It is enough to show the strongest case when $k=s-3$. We prove this by showing 
    \[
    2\sum_{i=0}^{2s-2}f_{n,i}-f_{n,2s+2}<0.
    \]
    By the same argument as in the Lemma \ref{3s}, it suffices to show
    \[
    2\frac{r^2}{1-r}-\frac{1}{r^2}<0.
    \]
    The left hand side is $-1.1144<0$.
    \item Next we show that $f_{n,2k-1}+2f_{n,2k}+f_{n,2k+1}+\cdots+f_{n,[(n+k)/2]}-f_{n,n-k}<0$ is true when $k=s-2$. The inequality becomes
    \[
    f_{n,2s-5}+2f_{n,2s-4}+f_{n,2s-3}+f_{n,2s-2}-f_{n,2s+1}<0.
    \]
    And it suffices to show
    \[
r^5+2r^4+r^3+r^2-\frac{1}{r}<0.
\]
The left hand side is $-0.8050<0$.
\end{enumerate}
\end{proof}

This completes the proof of the inequalities we will use in the case of $n=3s-1$. Next we consider the case $n = 3s - 2$.

\begin{lemma}\label{3s-2}
If $n=3s-2$, then we have $f_{n,2k-1}+2f_{n,2k}+f_{n,2k+1}+\cdots+f_{n,[(n+k)/2]}-f_{n,n-k}<0$ for $0<k<s-1$ and $n \gg 0$.
\end{lemma}

\begin{proof}\ 
\begin{enumerate}
    \item We first show $f_{n,2k-1}+2f_{n,2k}+f_{n,2k+1}+\cdots+f_{n,[(n+k)/2]}-f_{n,n-k}<0$ is true when $k<s-2$. It is enough to show the strongest case when $k=s-3$. We prove this by showing 
    \[
    2\sum_{i=0}^{2s-3}f_{n,i}-f_{n,2s+1}<0.
    \]
    By the same argument as in the Lemma \ref{3s}, it suffices to show
    \[
    2\frac{r^3}{1-r}-\frac{1}{r}<0.
    \]
    The left hand side is $-0.6609<0$.
    \item Next we show that $f_{n,2k-1}+2f_{n,2k}+f_{n,2k+1}+\cdots+f_{n,[(n+k)/2]}-f_{n,n-k}<0$ is true when $k=s-2$. The inequality becomes
    \[
    f_{n,2s-5}+2f_{n,2s-4}+f_{n,2s-3}+f_{n,2s-2}-f_{n,2s}<0.
    \]
    And it suffices to show
    \[
r^5+2r^4+r^3+r^2-1<0.
\]
The left hand side is $-0.1189<0$.
\end{enumerate}
\end{proof}

Now we are well prepared and next we proceed to find an optimal solution to the Linear Program and give a proof of Theorem~\ref{optimum}.

\section{An optimal solution to the linear program}

The key idea of the proof is by induction. The strategy is that we first prove if $(t_0, t_1,\cdots, t_{2n}) $ is a solution to the Linear Program, then we can find a better solution which satisfies $t_i = 0$ for all $i>n$. Assume now we have a solution with the property that $t_i = 0$ for all $i>n$, then we show that we can find a better solution such that $t_n = 0$. Then we proceed further and show that for $i<n$, we can still find a better solution with more $t_i$ equal to zero. Finally we get the claimed optimal solution.

\begin{prop}\label{2n}
If $(t_0, t_1,\cdots, t_{2n}) $ is a solution to the Linear Program, then we can find a better solution with the property that $t_{2n}=0$.
\end{prop}

\begin{proof}
Assume $t_{2n}=\epsilon>0$, the only constraint containing $t_{2n}$ is: 
\[
2t_0+t_{2n}=2t_0+\epsilon\ge 1.
\]
Let $t_0'=t_0+\epsilon/2$, $t_{2n}'=0$, $t_i'=t_i$ for $1\le i\le 2n-1$. Since there is only one constraint containing $t_{2n}$, which is $2t_0'+t_{2n}'\ge 1$, and we only change $t_0$ and $t_{2n}$, all the other constraints involving $t_0$ are satisfied since we increase $t_0$. Those constraints without $t_0,t_{2n}$ do not change. Therefore the set $(t_0', t_1',\cdots, t_{2n}') $ also gives a solution. However, $\sum_{i=0}^{2n}f_{n,i}t_i'-\sum_{i=0}^{2n}f_{n,i}t_i=f_{n,0}\epsilon/2-f_{n,2n}\epsilon<0$. Here we used the result $f_{n,0}=f_{n,2n}$. So we get a better solution.
\end{proof}

Next we improve the result in Proposition \ref{2n}.

\begin{prop}\label{>n}
Let $(t_1, t_2,\cdots, t_{2n})$ be a solution to the Linear Program obtained as in Proposition \ref{2n}, then we can find a better solution with the property that $t_{i}=0$ for all $i>n$.
\end{prop}

\begin{proof}
By Proposition \ref{2n}, we have $t_{2n} = 0$. Therefore a reasonable approach is to use induction. Let $s$ be an integer such that $s>n$, assume by induction that we have shown $t_l=0$ for all $l>s$, then we show that we can get a better solution with $t_s=0$. Assume $t_s=\epsilon>0$. The constraints are: $t_i+t_j+t_k\ge 1$ for $i+j+k\le 2n$. We only care about the constraints with a term $t_{s}$, say $k=s$, these constraints are $t_i + t_j + t_{s}\ge 1$, with $i+j\le 2n-s$. If $i+j<2n-s$, then $2n-i-j>s$ and $t_i+t_j+t_{2n-i-j}=t_i+t_j\ge 1$. Therefore we only need to consider the case $i+j=2n-s$. The constraints are 
\[
    t_i+t_{2n-s-i}+t_s\ge 1
\]
for $i=0,1,\cdots, [n-s/2]$. Where $[x]$ is the integral part of $x$. We change the $t_i$'s, let 
\[
    t_i'=t_i+\epsilon
\]
for $i=0,1,\cdots, [n-s/2]$ and $t_s'=0$, fix all other $t_j$'s. These new $t_i'$ satisfy the constraints, and we have $\sum_{i=0}^{2n}f_{n,i}t_i'-\sum_{i=0}^{2n}f_{n,i}t_i=\sum_{i=0}^{ [n-s/2]}f_{n,i}\epsilon-f_{n,s}\epsilon$. However, by Lemma \ref{1}, we have the following inequality

\begin{equation}
\sum_{i=0}^{ [n-s/2]}f_{n,i}-f_{n,s}<0.
\label{eq:P1}\end{equation}

So we get a better solution. 

\end{proof}

Next based on the result in Proposition \ref{>n}, we show that we actually can improve the result by letting $t_n = 0$, the situation is a little bit different since in this case we have more constraints. We analyze this in the following proposition:

\begin{prop}\label{n}
Let $(t_1, t_2,\cdots, t_{2n})$ be a solution to the linear program obtained as in Proposition \ref{>n}, then we can find a better solution with the property that $t_{n}=0$.
\end{prop}

\begin{proof}
Assume $t_n=\epsilon>0$. Among the indices $i, j ,k$, since $i+j+k\le 2n$, we can have at most two indices equal to $n$. If two of them are equal to $n$, say $i=j=n$, then the constraint is 
\[
t_0+2t_n\ge 1.
\]
If there is only one index equal $n$, say $k = n$, we have $t_i+t_j+t_n\ge 1$. If $i+j<n$, then $2n-i-j>n$ and $t_i+t_j+t_{2n-i-j}=t_i+t_j\ge 1$. Therefore we only need to consider the case when $i+j=n$. The constraints are
\begin{align*}
     &t_0+2t_n\ge 1,
     \\
     &t_1+t_{n-1}+t_n\ge 1,
     \\
     &\cdots
     \\
     &t_i+t_{n-i}+t_n\ge 1,
     \\
     &\cdots
     \\
     & t_{[n/2]}+t_{\{n/2\}}+t_n\ge 1. 
     \\
\end{align*}
Where $\{x\}$ denotes the least integer that is greater than or equal to $x$. Since we only increase $t_i$'s, the new $t_i$'s automatically satisfy the constraints without $t_n$ term. Let 

\begin{align*}
    &t_0'=t_0+2\epsilon,
    \\
    &t_1'=t_1+\epsilon,
    \\
    &\cdots
    \\
    &t_i'=t_i+\epsilon,
    \\
    &\cdots
    \\
    &t_{[n/2]}'=t_{[n/2]}+\epsilon,
    \\
    &t_n'=0.
\end{align*}
All the other $t_i$'s are not changed. As we see, these new $t'_i$ satisfy all the constraints and $\sum_{i=0}^{2n}f_{n,i}t_i'-\sum_{i=0}^{2n}f_{n,i}t_i=\epsilon(2f_{n,0}+\sum_{i=1}^{[n/2]}f_{n,i}-f_{n,n})$. By Lemma~\ref{2}, we have

\begin{equation}
2f_{n,0}+\sum_{i=1}^{[n/2]}f_{n,i}-f_{n,n}<0.
\label{eq:P2}\end{equation}

Therefore we have a better solution. 
\end{proof}

We will keep using the above strategy. Assume we can proceed the above argument, which means for some integer $l$ with $l<n$, we have $t_{s}=0$ for all $s>l$, we analyze the condition under which we can find a better solution with $t_l=0$. Let $l=n-k$. For $i\le 2k-2$, we have $i+2(n-k+1)\le 2n$, so $t_{i}+2t_{n-k+1}\ge 1$, and $t_{i}\ge 1$ since $t_{n-k+1}=0$. By the same argument as in the proposition above  and the fact that $t_{i}\ge 1$ for all $i\le 2k-2$, we only need to consider the constraints $t_i+t_j+t_d\ge 1$ with a term $t_l$ and $i, j, d >2k-2$, where $l=n-k$. Say $d=l$, then $t_i+t_j+t_l\ge 1$, and $i+j\le 2n-l=n+k$. If $i+j<2n-l=n+k$, then $t_i+t_j+t_{2n-i-j}=t_i+t_j\ge 1$. Therefore we only need to consider the case $i+j+l=2n$. The constraints are
\begin{align*}
    &t_{2k-1}+t_{n-k}+t_{n-k+1}\ge 1,
    \\
    &t_{2k}+2t_{n-k}\ge 1,
    \\
    &t_{2k+1}+t_{n-k-1}+t_{n-k}\ge 1,
    \\
    &\cdots
    \\
    &t_{2k+r}+t_{n-k-r}+t_{n-k}\ge 1,
    \\
    &
    \cdots
    \\
    &t_{[(n+k)/2]}+t_{\{(n+k)/2\}}+t_{n-k}\ge 1.
\end{align*}

Assume $t_l = \epsilon > 0$, we change the $t_i$'s by the following rule
\begin{align*}
    &t_{2k-1}'=t_{2k-1}+\epsilon,
    \\
    &t_{2k}'=t_{2k}+2\epsilon,
    \\
    &t_{2k+1}'=t_{2k+1}+\epsilon,
    \\
    &\cdots
    \\
    &t_{[(n+k)/2]}'=t_{[(n+k)/2]}+\epsilon,
    \\
    &t_{n-k}'=0.
\end{align*}
All the other $t_i$'s are not changed. Then we have 
\begin{align*}
    &\sum_{i=0}^{2n}f_{n,i}t_i'-\sum_{i=0}^{2n}f_{n,i}t_i
    \\
    &=\epsilon(f_{n,2k-1}+2f_{n,2k}+f_{n,2k+1}+\cdots+f_{n,[(n+k)/2]}-f_{n,n-k}).
\end{align*}

If the following inequality holds, we would have a better solution:
\begin{equation}
f_{n,2k-1}+2f_{n,2k}+f_{n,2k+1}+\cdots+f_{n,[(n+k)/2]}-f_{n,n-k}<0.
\label{P3}\end{equation}

As we pointed out, the solution depends on $n$ mod $3$. Recall that $l = n - k$, if $n = 3s$, from Lemma \ref{3s}, the above inequality is true for $0<k<s$; if $n=3s-1$, Lemma \ref{3s-1} tells us that the above inequality is true if $0<k<s-1$; if $n = 3s-2$, Lemma \ref{3s-2} tells us that the above inequality holds for $0<k<s-1$. For convenience, we list the results for which the above inequality holds in three cases.

\begin{prop}\ 
\begin{enumerate}
    \item If $n=3s$, then the inequality \eqref{P3} holds for $0<k<s$, therefore an optimal solution to the Linear Program in Section \ref{L} satisfies $t_{i} = 0$ for $i \ge n-(s-1)=2s+1$.

    \item If $n=3s - 1$, then the inequality \eqref{P3} holds for $0<k<s-1$, therefore an optimal solution to the Linear Program in Section \ref{L} satisfies $t_{i} = 0$ for $i \ge n-(s-2)=2s+1$.

    \item If $n=3s-2$, then the inequality \eqref{P3} holds for $0<k<s-1$, therefore an optimal solution to the Linear Program in Section \ref{L} satisfies $t_{i} = 0$ for $i \ge n-(s-2)=2s$.

\end{enumerate}
\end{prop}

We can not proceed further, otherwise these inequalities are not true. By the above analysis, we can greatly reduce the number of constraints. Finally the number of variables in the constraints are $2,3,4$ respectively in the three cases, which can be computed by hand. We show that the optimal solutions are:
\begin{corollary}
If $n=3s$, $t_i=1$ for $0\le i\le 2s-2$, $t_{2s-1}=\frac{2}{3}, t_{2s}=\frac{1}{3}$, and $t_i=0$ for $i\ge 2s+1$.
\end{corollary}

\begin{proof}

We have $t_{l}=0$ for $l\ge 2s+1$. Since $t_i+2t_{2s+1}\ge 1$ for $i\le 2s-2$, we have $t_{i}\ge 1$ for $i\le 2s-2$, and by $3t_{2s}\ge 1$, we have $t_{2s}\ge \frac{1}{3}$. The constraints containing $t_{2s-1}$ is $t_{2s-1}+t_{2s}+t_{2s+1}\ge 1$, which is $t_{2s-1}+t_{2s}\ge 1$. If $t_{2s}=\frac{1}{3}+\delta$, then $t_{2s-1}\ge \frac{2}{3}-\delta$. However, $f_{n,2s}>f_{n,2s-1}$. So to minimize $\sum_{i=0}^{2n}f_{n,i}t_i$, we must have $\delta=0$, and $t_{2s}=\frac{1}{3}, t_{2s-1}=\frac{2}{3}$.
\end{proof}

\begin{corollary}
 If $n=3s-1$, $t_i=1$ for $0\le i\le 2s-4$, $t_{2s-3}=\frac{4}{5}, t_{2s-2}=\frac{3}{5},t_{2s-1}=\frac{2}{5}, t_{2s}=\frac{1}{5}$, and $t_i=0$ for $i\ge 2s+1$.
\end{corollary}

\begin{proof}
For $i\le 2s-4$, we have $t_{i}+2t_{2s+1}=t_{i}\ge 1$. For $i\ge 2s+2, t_{i}=0$. Let $x=t_{2s-3},y=t_{2s-2},z=t_{2s-1},w=t_{2s}$, the relevant inequalities are 
\[
2z+w\ge 1,\quad 2y+w\ge 1,\quad 2x+w\ge 1,
\]
\[
y+2w\ge 1,\quad x+2w\ge1,
\]
\[
x+w\ge 1,
\]
\[
y+z\ge 1,\quad x+z\ge 1,\quad 3z\ge 1,
\]
\[
x+y\ge 1,\quad 2y\ge 1,\quad 2x\ge 1,
\]
Some of the inequalities are redundant, the effective inequalities are
\[
2z+w\ge 1,\quad 2y+w\ge 1,\quad y+2w\ge 1,\quad x+w\ge 1,\quad y+z\ge 1,\quad x+z\ge 1,
\]
\[
x\ge \frac{1}{2},\quad y\ge \frac{1}{2},\quad z\ge \frac{1}{3},
\]
Since we have $4$ variables, therefore, the optimal solution must make four of those inequalities be strictly equalities. Since we have $2y+w\ge 1, y+2w\ge 1$ and $y+z\ge 1,x+z\ge 1$, we can divide into the following cases: $y>w$ or $y<w$, and $x<y$ or $x>y$. It turns out that only $y>w, x>y$ works, then it satisfies
\[
2z+w=1,\quad y+2w=1,\quad x+w=1,\quad y+z=1,
\]
which is $x=\frac{4}{5},y=\frac{3}{5},z=\frac{2}{5},w=\frac{1}{5}$.
\end{proof}

\begin{corollary}
 If $n=3s-2$, $t_i=1$ for $0\le i\le 2s-4$, $t_{2s-3}=\frac{3}{4}, t_{2s-2}=\frac{2}{4},t_{2s-1}=\frac{1}{4},$ and $t_i=0$ for $i\ge 2s$.
\end{corollary}

\begin{proof}
The same argument as above works.
\end{proof}

We have found an optimal solution and completed the proof of Theorem \ref{optimum}. With Theorem~\ref{optimum} and Corollary~\ref{estimate}, as well as the fact that $c(n)\le \rk^G(v)\le \sum_{i=0}^{2n}f_{n,i}t_i$, we can prove Theorem \ref{capset}. Recall that $c(n)$ is the largest size of a cap set in $\mathbb F^n_3$.

\textbf{Proof of Theorem \ref{capset}}:

In Corollary \ref{estimate}, let $\alpha = 2/3$, it is easy to compute that $r = \frac{\sqrt{33}-1}{8}$, and $f(r) = \frac{1+r+r^2}{r^{2/3}}\approx 2.7551$.

\begin{enumerate}
    \item  If $n = 3s$ for some integer $s > 0$, we now show that
     \begin{equation}
         c(n) \le 2.4951\frac{f(r)^n}{\sqrt{n}}\left(1+o(1)\right) = O\left(\frac{f(r)^n}{\sqrt{n}}\right).
     \end{equation}
     By Theorem \ref{optimum}, we get $t_i=1$ for $i=0,1,\dots,2s-2$,
     $t_{2s - 1} = \frac{2}{3}$, $t_{2s} = \frac{1}{3}$, and $t_i = 0$ for $i \ge 2s + 1$. Therefore we have
    \[
    \sum_{i = 0}^{2n}f_{n,i}t_i = \sum_{i = 0}^{2s-2}f_{n,i}+\frac{2}{3}f_{n,2s-1}+\frac{1}{3}f_{n,2s}.
    \]
    And by Corollary \ref{estimate}, we have
    \[
    \sum_{i=0}^{2n}f_{n,i}t_i=\frac{f(r)^n}{\sqrt{2\pi n}}\sqrt{\frac{1+r+r^2}{2\alpha-(1-\alpha)r}}\Big(\frac{1}{(1-r)r^{-2}}+\frac{2}{3r^{-1}}+\frac{1}{3}\Big)(1 + o(1))
    \]
    \[
    \approx 0.8371\frac{f(r)^n}{\sqrt{n}}(1+o(1)).
    \]
   Therefore we get
   \[
   c(n)\le 3\sum_{i = 0}^{2n}f_{n,i}t_i \le 2.4951\frac{f(r)^n}{\sqrt{n}}\left(1+o(1)\right) = O\left(\frac{f(r)^n}{\sqrt{n}}\right).
   \]
   
   \item  If $n = 3s - 1$ for some integer $s > 0$, we now show that
     \begin{equation}
         c(n) \le 1.7529\frac{f(r)^n}{\sqrt{n}}\left(1+o(1)\right) = O\left(\frac{f(r)^n}{\sqrt{n}}\right).
     \end{equation}
    By Theorem \ref{optimum}, we get $t_i=1$ for $i=0,1,\dots,2s-4$, $t_{2s - 3} = \frac{4}{5}$, $t_{2s-2} = \frac{3}{5}$, $t_{2s-1}= \frac{2}{5}$, $t_{2s}=\frac{1}{5}$ and $t_i = 0$ for $i \ge 2s + 1$. Therefore we have
    \[
    \sum_{i = 0}^{2n}f_{n,i}t_i = \sum_{i = 0}^{2s-4}f_{n,i}+\frac{4}{5}f_{n,2s-3}+\frac{3}{5}f_{n,2s-2}+\frac{2}{5}f_{n,2s-1}+\frac{1}{5}f_{n,2s}.
    \]
    And by Corollary \ref{estimate}, we have
    \[
    \sum_{i=0}^{2n}f_{n,i}t_i=\frac{f(r)^n}{\sqrt{2\pi n}}\sqrt{\frac{1+r+r^2}{2\alpha-(1-\alpha)r}}\Big(\frac{1}{(1-r)r^{-4}}+\frac{4}{5r^{-3}}+\frac{3}{5r^{-2}}+\frac{2}{5r^{-1}}+\frac{1}{5}\Big)\big(1 + o(1)\big)
    \]
    \[\approx 0.5843\frac{f(r)^n}{\sqrt{n}}\big(1 + o(1)\big).
    \]
   Therefore we get
   \[
   c(n)\le 3\sum_{i = 0}^{2n}f_{n,i}t_i \le 1.7529\frac{f(r)^n}{\sqrt{n}}\left(1+o(1)\right) = O\left(\frac{f(r)^n}{\sqrt{n}}\right).
   \]
   
   \item If $n = 3s - 2$ for some integer $s > 0$, we now show that
     \begin{equation}
         c(n) \le 1.2288\frac{f(r)^n}{\sqrt{n}}\left(1+o(1)\right) = O\left(\frac{f(r)^n}{\sqrt{n}}\right).
     \end{equation}
     By Theorem \ref{optimum}, we get $t_i = 1$ for $i = 0, 1, \cdots, 2s - 4$, $t_{2s - 3} = \frac{3}{4}$, $t_{2s-2} = \frac{2}{4}$, $t_{2s-1}= \frac{1}{4}$, and $t_i = 0$ for $i \ge 2s$. Therefore we have
    \[
    \sum_{i = 0}^{2n}f_{n,i}t_i = \sum_{i = 0}^{2s-4}f_{n,i}+\frac{3}{4}f_{n,2s-3}+\frac{2}{4}f_{n,2s-2}+\frac{1}{4}f_{n,2s-1}.
    \]
    And by Corollary \ref{estimate}, we have
    \[
    \sum_{i=0}^{2n}f_{n,i}t_i=\frac{f(r)^n}{\sqrt{2\pi n}}\sqrt{\frac{1+r+r^2}{2\alpha-(1-\alpha)r}}\Big(\frac{1}{(1-r)r^{-4}}+\frac{3}{4r^{-3}}+\frac{2}{4r^{-2}}+\frac{1}{4r^{-1}}\Big)\big(1 + o(1)\big)
    \]
    \[\approx 0.4096\frac{f(r)^n}{\sqrt{n}}\big(1 + o(1)\big).
    \]
   Therefore we get
   \[
   c(n)\le 3\sum_{i = 0}^{2n}f_{n,i}t_i \le 1.2288\frac{f(r)^n}{\sqrt{n}}\left(1+o(1)\right) = O\left(\frac{f(r)^n}{\sqrt{n}}\right).
   \]
\end{enumerate}

\begin{comment}
\section{Conclusion}

The $G$-stable rank for tensors introduced in \cite{Derksen} is a new notion of rank for tensors, it has some properties as usual rank for tensors, but it could be rational numbers.

Combining the G-stable rank for tensors with carefully estimate the coefficients of polynomials give us improved upper bounds for the Capset problem.

The application of $G$-stable rank to the Cap Set Problem gives an upper bound for the size of cap set, i.e. $c(n)\le O\big(\frac{f(r)^n}{\sqrt{n}}\big)$ for $f(r)=\frac{1+r+r^2}{r^{2/3}}$ and $r = \frac{\sqrt{33}-1}{8}$. The coefficient of the upper bound depends  on $n \mod 3$, as proved in Corollary~\ref{capset}. We showed this by giving an estimation of $f_{n,i}$, the coefficient of $x^i$ in the polynomial $(1+x+x^2)^n$, and solving the linear program \ref{L}, which also proved the conjecture 6.1 in \cite{Derksen}. We also give a good estimate of the coefficients for the polynomial $(1+x+x^2+\cdots+x^{q-1})^n$ for some positive integer $q>0$, as a result, we get an upper bound for the size of subset in $\mathbb F^n_q$ with no three-term arithmetic progression. Our upper bound is based on \cite{EllenbergGijswijt} but is more precise.

\end{comment}

\section{Acknowledgement}
 I would like to thank Harm Derksen, for introducing me to this topic and for all of his support and discussion.


\begin{thebibliography}{9}

\bibitem{BatemanKatz}
Michael Bateman, Nets H. Katz, 
\textit{New bounds on cap sets},
{J. Amer. Math. Soc. 25 (2012), no. 2, 585–613.}


\bibitem{BrownBuhler}
Tom C. Brown, Joe P. Buhler,
\textit{A density version of a geometric Ramsey theorem},
{J. Combin. Theory Ser. A 32 (1982), no. 1, 20–34.}

\bibitem{CrootLevPach}
Ernie Croot, Vsevolod Lev, Peter Pach, 
\textit{Progression-free sets in $\Z/4\Z$ are exponentially small},
Annals of Math. 185 (2017), no. 1, 331--337.


\bibitem{Derksen}
Harm Derksen, 
\textit{The G-stable rank for tensors}, 
arXiv:2002.08435.


\bibitem{Edel}

Yves Edel,
\textit{Extensions of generalized product caps},
{Designs, Codes and Cryptography 31, (2004), 5–14.}


\bibitem{EllenbergGijswijt}
Jordan S. Ellenberg, Dion Gijswijt,
\textit{On large subsets of $\mathbb{F}^n_q$ with no three-term arithmetic progression},
{Ann. of Math. (2) 185 (2017), no. 1, 339–343.}



\bibitem{KSS} 
Robert Kleinberg, Will Sawin, David E. Speyer,
\textit{The growth of tri-colored sum-free sets}, 
{Discrete Anal. (2018), Paper No. 12.}





\bibitem{Meshulam}
Roy Meshulam,
\textit{On subsets of finite abelian groups with no 3-term arithmetic progressions},
{J. Combin. Theory Ser. A 71 (1995), no. 1, 168–172.}




\bibitem{Tao}
Terence Tao,
\textit{A symmetric formulation of the Croot–Lev–Pach–Ellenberg–Gijswijt capset bound (2016)},
{available at https://terrytao.wordpress.com/2016/05/18/. blog post}

\bibitem{Tyrrell}
Fred Tyrrell,
\textit{New lower bounds for cap sets},
{arXiv:2209.10045}


\end{thebibliography}
\end{document}